\documentclass[12pt]{amsart}
\usepackage{amsmath,amsfonts,amsthm,amssymb,xypic}
\newcommand{\scrA}{\mathcal{A}}
\newcommand{\scrB}{\mathcal{B}}
\newcommand{\scrC}{\mathcal{C}}

\newcommand{\scrP}{\mathcal{P}}
\newcommand{\scrS}{\mathcal{S}}

\newcommand{\bA}{\mathbf{A}}

\newcommand{\bE}{\mathbf{E}}
\newcommand{\bF}{\mathbf{F}}
\newcommand{\bP}{\mathbf{P}}

\newcommand{\bR}{\mathbf{R}}

\newcommand{\bbF}{\mathbb{F}}
\newcommand{\bbR}{\mathbb{R}}
\newcommand{\bbZ}{\mathbb{Z}}

\newcommand{\Hom}{\mathrm{Hom}}
\newcommand{\Ext}{\mathrm{Ext}}
\newtheorem{dummy}{Dummy}[section]
\newtheorem{proposition}[dummy]{Proposition}
\newtheorem{lemma}[dummy]{Lemma}

\newtheorem{theorem}[dummy]{Theorem}
\theoremstyle{definition}
\newtheorem{definition}[dummy]{Definition}
\newtheorem{remark}[dummy]{Remark}
\newtheorem{example}[dummy]{Example}

\newcommand{\twocolim}{2\!\varinjlim}

\newcommand{\shHom}{\underline{\mathrm{Hom}}}
\newcommand{\tilt}{\mathrm{tilt}}
\title{Stacks similar to the stack of perverse sheaves}
\author{David Treumann}
\date{January 2008}

\begin{document}
\maketitle

\begin{abstract}
We introduce, on a topological space $X$, a class of stacks of abelian categories we call ``stacks of type P.''  This class of stacks includes the stack of perverse sheaves (of any perversity, constructible with respect to a fixed stratification), and is singled out by fairly innocuous axioms.  We show that some basic structure theory for perverse sheaves holds for a general stack of type P: such a stack is locally equivalent to a MacPherson-Vilonen construction, and under certain connectedness conditions its category of global objects is equivalent to the category of modules over a finite-dimensional algebra.  To prove these results we develop a rudimentary tilting formalism for stacks of type P -- another sense in which these stacks are ``similar to stacks of perverse sheaves.''
\end{abstract}

\section{Introduction}
The abelian category $\bP = \bP(X,\scrS)$ of perverse sheaves on a space $X$, constructible with respect to a stratification $\scrS$, was introduced in \cite{bbd} and has a raft of applications, especially in representation theory.  A standard complaint about perverse sheaves is that they are not very concrete: $\bP$ is defined as a subcategory of the derived category of sheaves on $X$, but many of its basic properties are not obvious from this definition.  Moreover when one wishes to work with a particular perverse sheaf it is not very efficient, if it is even possible, to write it down as a chain complex.  

A lot of work has been done to address this complaint.  One approach has been to develop methods to describe $\bP(X,\scrS)$ more explicitly, both in general (e.g. \cite{macphersonvilonen}, \cite{gmvquiver}) and for especially interesting pairs $(X,\scrS)$ (e.g. \cite{bradengrinberg}, \cite{braden}).  Another approach (e.g. \cite{mirollovilonen}, \cite{cps}) has been to try to understand more abstractly ``what kind'' of category $\bP(X,\scrS)$ is -- for instance (especially in \cite{cps}) by formulating the useful properties of perverse sheaves as axioms rather than as consequences of an opaque definition.
This paper is a contribution to this circle of ideas.  Our starting point is the fact (proved already in \cite{bbd}) that $\bP(X,\scrS)$ is the category of global objects of a stack of categories $\scrP$ on $X$ -- this means roughly that a perverse sheaf on $X$ may be given by a perverse sheaf on each chart of an open cover, together with descent data.  We are interested in ``what kind'' of stack $\scrP$ is, rather than ``what kind'' of category $\bP$ is.  

In this paper, we introduce a class of stacks we call ``stacks of type P.''  The stacks of type P include the stack of perverse sheaves (of any perversity), and they are singled out by a list of fairly innocuous axioms.  The most basic of these is that a stack of type P should be constructible with respect to a stratification of $X$.  (The notion of a constructible stack, directly analogous to the notion of a constructible sheaf, is introduced in \cite{epcs}).  We investigate the extent to which these stacks behave ``similarly to stacks of perverse sheaves.''

Let us briefly recall the ``elementary construction'' of $\bP$ given in \cite{macphersonvilonen}.  (We review it in more detail in section \ref{section4}.)
Given an abelian category $\bA$, functors $F$ and $G$ from $\bA$ to vector spaces, and a natural transformation $T:F \to G$, MacPherson and Vilonen constructed a new category $C(F,G;T)$ whose objects and morphisms are given explicitly in terms of objects of $\bA$, vector spaces, and linear maps.  On a stratified space of the form $CL \times \bbR^k$, where $L$ is a compact stratified space and $CL$ is the open cone on $L$, they showed that $\bP(CL \times \bbR^k)$ is equivalent to a category of the form $C(F,G;T)$, where $F$ and $G$ are functors on $\bP(L)$.

On a general stratified space $X$ this result may be interpreted as a statement about stalks of the stack $\scrP$: every point $x \in X$ has a neighborhood of the form $CL \times \bbR^k$, and the restriction functor $\bP(CL \times \bbR^k) \to \scrP_x$ is an equivalence (this is equivalent to the constructibility of $\scrP$), so the stalks of $\scrP$ are equivalent to MacPherson-Vilonen constructions.  Our main result is that this holds for an arbitrary stack $\scrC$ of type P:  the stalks of $\scrC$ are equivalent to MacPherson-Vilonen constructions.  To establish this we develop a rudimentary ``tilting formalism'' (as in \cite{bbm}) for stacks of type P, another sense in which these stacks are similar to stacks of perverse sheaves.

We use this result and some basic properties of constructible stacks to obtain a finiteness theorem, following an argument in \cite{mirollovilonen}: if $\scrC$ is a stack of type P on a stratified space $X$ all of whose strata are 2-connected, then the category $\scrC(X)$ is equivalent to the category of modules over a finite-dimensional algebra.

\subsection{Statement of results}

Fix an algebraically closed field $\bbF$.  Let us call a stack $\scrC$ of categories on a space $X$ an $\bbF$-linear \emph{abelian stack} if each category $\scrC(U)$ is an $\bbF$-linear abelian category, and if each of the restriction functors $\scrC(U) \to \scrC(V)$ is an exact $\bbF$-linear functor.  

For our conventions on stratification theory see section \ref{conventions}.

\begin{definition}
\label{defPstacks}
A stack $\scrC$ on a topologically stratified space $(X,\scrS)$ is a \emph{$\bbF$-linear stack of type P} if it satisfies the following conditions:
\begin{enumerate}
\item $\scrC$ is constructible with respect to $\scrS$, in the sense of \cite{epcs}.  Recall that this means that the restriction of $\scrC$ to each stratum of $\scrS$ is locally constant.
\item $\scrC$ is an $\bbF$-linear abelian stack.  Each stalk category $\scrC_x$ has the property that all its objects are of finite-length, and that all the vector spaces $\Hom(c,d)$ and $\Ext^1(c,d)$ are finite-dimensional.
\item In each stalk category $\scrC_x$, there is a unique simple object $s_x$ supported on the stratum containing $x$.  Furthermore, the Yoneda ext groups $\Ext^1(s_x,s_x)$ and $\Ext^2(s_x,s_x)$ vanish.
\item If $Z \subset X$ is a closed union of strata, and $U \subset X$ is any open set, then the restriction functor $\scrC(U) \to \scrC(U - U \cap Z)$ exhibits the latter category as a Serre quotient of $\scrC(U)$ by the subcategory $\scrC_Z(U)$ of objects supported on $Z$ (definition \ref{twonewstacks}).
\end{enumerate}
\end{definition}

We prove the following proposition in section \ref{sec2}:

\begin{proposition}
\label{scrPistypeP}
Let $(X,\scrS)$ be a topologically stratified space, and let $p:\scrS \to \bbZ$ be any function from connected strata of $\scrS$ to integers.  The stack $\scrP = \scrP_{\scrS,p}$ of $\scrS$-constructible $p$-perverse sheaves on $X$ is a stack of type P.
\end{proposition}

Let us say that $\scrC$ \emph{has recollement operations} for the stratification if one can make sense of the usual sheaf operations $i^*,i^!,i_*,i_!$ in the following sense: whenever $Z$ is a closed union of strata and $U$ is an open set, the inclusion $\scrC_Z(U) \hookrightarrow \scrC(U)$ and the restriction
$\scrC(U) \to \scrC(U - U \cap Z)$ admit adjoints on both sides.  Here $\scrC_Z(U)$ denotes the full subcategory of $\scrC(U)$ whose objects are supported on $Z$.

\begin{theorem}
\label{maintheorem}
Let $(X,\scrS)$ be a topologically stratified space, and let $\scrC$ be an $\scrS$-constructible $\bbF$-linear stack of type P on $X$.
\begin{enumerate}
\item $\scrC$ has recollement operations for the stratification.
\item Each stalk $\scrC_x$ is equivalent to a MacPherson-Vilonen construction.  More specifically, if $U$ is a regular neighborhood of $x$ and $Z \subset U$ is the closed stratum containing $x$, then $\scrC_x$ is equivalent to a MacPherson-Vilonen construction $C(F,G;T)$, where $F$ and $G$ are functors $\scrC(U - Z) \to \bbF\text{-Vect}$.
\item Suppose that $\scrS$ is a Thom-Mather stratification.  If all the strata are 2-connected, then the category of global objects is equivalent to the category of finite-dimensional modules over a finite-dimensional $\bbF$-algebra.
\end{enumerate}
\end{theorem}

\subsection{Further problems}  Let us discuss some open questions related to the contents of this paper.

\emph{Problem 1.}  Let $X$ be a complex algebraic variety and let $\scrS$ be a Whitney stratification of $X$ by complex algebraic subsets.  Can one characterize the stack $\scrP$ of $\scrS$-constructible, middle-perversity perverse sheaves on $X$, by a list of simple axioms?  This problem was formulated to address the ``standard complaint'' discussed above, and motivated the definition of ``stack of type P,'' but at some point I became skeptical that it has a nice answer.

Note that the most interesting and beautiful properties of perverse sheaves are consequences of complex geometry or Hodge theory; a closely related question is whether these properties entail anything interesting about the stack $\scrP$.  Our ``type P'' abstraction ignores these complex algebraic features of perverse sheaves, to the point that we found it better to develop the theory in a purely topological setting.

\emph{Problem 2.}  Can the usual sheaf-theory package (six operations, duality) be extended or modified to the setting of stacks of abelian categories?  (More plausibly, to a more high-tech setting such as stacks of dg categories or stable $\infty$-categories.)  This paper and \cite{epcs} constitute, in part, an effort to start developing a structure theory for such stacks.  See for instance remark \ref{rem-twonewstacks}.

\emph{Problem 3.}  In general, is the category of global objects $\scrC(X)$ of a stack of type P equivalent to the category of finite-dimensional modules over a finitely-presented algebra?  This is a theorem of \cite{gmvquiver} when $\scrC$ is the stack of perverse sheaves, but it is proved by constructing a very explicit presentation for a ring, by maneuvers (such as the Fourier transform for sheaves) not available in our setting.  In some sense the question is whether this finiteness result holds for a ``less interesting'' reason.  The major obstacle to using the techniques of this paper to resolve this is the fact that one cannot tell from the category of finite-dimensional $R$-modules whether $R$ is finitely presented.

\subsection{Notation and conventions}
\label{conventions}

We work throughout over an algebraically closed field $\bbF$.  If $a$ and $b$ are objects of an abelian category (always an $\bbF$-linear abelian category), then we will use $\Ext^n(a,b)$, for $n>0$, to denote the Yoneda ext group of equivalence classes of $(n+2)$-term exact sequences starting with $a$ and ending with $b$.  

A \emph{prestack} on a space $X$ is a weak functorial assignment from the open subsets of $X$ to categories -- that is, a 2-functor from the open subsets of $X$ ordered by reverse inclusion to the 2-category of categories.  A \emph{stack} on $X$ is a prestack satisfying a local-to-global condition.  We refer to \cite{epcs} for definitions.

We say that a space is \emph{2-connected} if it is connected and has vanishing first and second homotopy groups.

We need to make some remarks about stratification theory.  In this paper, we work in the very general setting of \emph{topologically stratified spaces}, introduced in \cite{ih2}.  The exception is in section \ref{secglobal}, where we introduce another technical condition -- always satisfied in practice, for instance on Thom-Mather stratified spaces -- to our stratified spaces (that strata, not just points, have regular neighborhoods).  A topological stratification $\scrS$ of a space $X$ is a decomposition into locally closed strata, with the property that each point has a neighborhood -- called \emph{regular} or \emph{conical} -- which is homeomorphic to $\bbR^k \times CL$ in a stratum-preserving way.  Here $L$ is a compact topologically stratified space of lower dimension, and $CL$ denotes the open cone on $L$ along with its induced decomposition.  We refer to \cite{ih2} and \cite{epcs} for more detailed definitions.

\section{Stacks of perverse sheaves}
\label{sec2}

In this section fix a topologically stratified space $(X,\scrS)$ and a function $p:\scrS \to \bbZ$ from connected strata of $\scrS$ to integers.  Let $\scrP$ denote the stack of $\scrS$-constructible $p$-perverse sheaves on $X$.  The purpose of this section is to prove proposition \ref{scrPistypeP}, i.e. that $\scrP$ is a stack of type P.  

We have to show $\scrP$ satisfies conditions (1) through (4) of definition \ref{defPstacks}.  That $\scrP$ is $\scrS$-constructible is proved in \cite{epcs}, %theorem 4.4
 so (1) holds, and that (4) holds is proved in \cite{bbd}.  It remains to prove (2) and (3).

\begin{proposition}
\label{2holds}
The stack $\scrP$ satisfies condition (2) of \ref{defPstacks}.  That is, $\scrP$ is an $\bbF$-linear abelian stack, and for each point $x \in X$, the stalk category $\scrP_x$ has only finite-length objects and finite-dimensional $\Hom$- and $\Ext^1$-spaces.
\end{proposition}

\begin{proof}
Since $\scrP$ is constructible, the stalk category $\scrP_x$ is equivalent to the category $\scrP(U)$ via the restriction map $\scrP(U) \to \scrP_x$, where $U$ is a regular neighborhood of $x$.  %theorem 3.13
The space $U$ is homeomorphic to $\bbR^k \times CL$, where $CL$ is the open cone on a compact  stratified space $L$.  Since $L$ is compact, the induced stratification on $U$ has finitely many strata.  We will prove that every object of $\scrP(U)$ has finite length by induction on the number of strata.  

Let $U' \subset U$ be an open stratum.  Any perverse sheaf $P \in \scrP(U)$ fits into an exact sequence
$${^p \! j}_! j^! P \to P \to P'' \to 0$$
where $P''$ is supported on the complement $Z$ of $U'$.  The space $Z$ is homeomorphic to $\bbR^j \times CM$ for some closed $M \subset L$, so we may assume $P''$ is of finite length by induction.  The sheaf $j^! P$ is a local system on $U'$ with finite-dimensional fibers, and ${^p \! j}_! j^!P$ fits into an exact sequence
$$0 \to Q \to {^p \! j}_! j^! P \to j_{!*} j^! P \to 0$$
where $Q$ is supported on $Z$.  Thus, ${^p \! j}_! j^! P$ is of finite length, and so $P$ is of finite length also.

For $n = 0,1$, the natural map $\Ext^n_{\scrP(U)}(P,Q) \to \Hom_{D^b(U)}(P,Q[n])$ is an isomorphism, so to show that $\scrP(U)$ has finite-dimensional $\Hom$ and $\Ext^1$ spaces it suffices to show that
$\Hom_{D^b(U)}(A,B)$ is finite-dimensional for all cohomologically $\scrS$-constructible objects of $D^b(U)$. We have $\Hom(A,B) \cong H^0(U;\bR\shHom(A,B))$, so by d\'evissage it suffices to note that the cohomology of an $\scrS$-constructible sheaf on $U \cong \bbR^k \times CL$ is finite-dimensional, as $L$ is compact.  % \cite{ih2} theorem on page 84. 
\end{proof}

\begin{proposition}
\label{3holds}
The stack $\scrP$ satisfies condition (3) of \ref{defPstacks}.  That is, each stalk category $\scrP_x$ has a unique simple object $s_x$ supported on the stratum containing $x$, and the Yoneda ext groups $\Ext^1(s_x,s_x)$ and $\Ext^2(s_x,s_x)$ vanish.
\end{proposition}

\begin{proof}
As in the proof of proposition \ref{2holds}, let $U$ be a regular neighborhood of $x$, so that $\scrP(U) \to \scrP_x$ is an equivalence.  Let $Z \subset U$ be the stratum containing $x$, and let $i:Z \hookrightarrow U$ be the inclusion map.  The unique simple object supported on $Z$ is $i_* V$, where $V$ is a 1-dimensional $\bbF$-vector space.  We have $\Ext^1(i_*V,i_*V) \cong \Hom_{D^b(U)}(i_* V,i_* V[1]) \cong H^1(Z;\bbF) = 0$.  Let
$$0 \to i_*V \to A \to B \to i_* V \to 0$$
be a Yoneda class in $\Ext^2(i_*V,i_*V)$.  To see that it vanishes we need to find an object $M \in \scrP(U)$ together with a filtration $i_*V \subset A \subset M$ such that $M/i_* V \cong B$ and $M/A \cong i_*V$.  Let $C$ denote the image of $A$ in $B$.  The sequence $0 \to i_*V \to A \to C \to 0$ induces an exact sequence
$$\Hom(i_*V,i_*V[1]) \to \Hom(i_*V,A[1]) \to \Hom(i_*V,C[1]) \to \Hom(i_*V,i_*V[2])$$
of hom spaces in $D(U)$.  The left- and right-hand groups are cohomology groups of $\bbR^k$, which vanish, and the middle groups are $\Ext^1(i_*V,A)$ and $\Ext^1(i_*V,C)$; thus $\Ext^1(i_* V,A) \cong \Ext^1(i_*V,C)$.  We may take $M$ to be a representative in $\Ext^1(i_*V,A)$ corresponding to $B \in \Ext^1(i_*V,C)$.
\end{proof}

\section{Recollement operations}

In this section, we prove part (1) of theorem \ref{maintheorem}.  That is, we show that a stack of type P has recollement operations.  Let us discuss what this means in more detail.  If $\scrC$ is an abelian stack on $X$, and $c$ is an object of $\scrC(U)$, define the \emph{support} of $c$ to be the set of all $z \in U$ such that the restriction of $c$ to the stalk $\scrC_z$ does not vanish.  If $Z$ is a closed subset of $X$ we may define two new stacks $\scrC_Z$ and $\scrC_{X - Z}$ on $X$:

\begin{definition}
\label{twonewstacks}
Let $Z$ be a closed subset of $X$, and let $\scrC$ be an abelian stack on $X$.
\begin{enumerate}
\item Denote by $\scrC_Z$ the prestack that maps an open set $U \subset X$ to the full subcategory $\scrC_Z(U) \subset \scrC(U)$ of objects that vanish under the restriction functor $\scrC(U) \to \scrC(U-U \cap Z)$.

\item Denote by $\scrC_{X-Z}$ the prestack that maps an open set $U \subset X$ to the category
$\scrC\big(U \cap (X-Z)\big)$.  
\end{enumerate}
It is easy to verify that these prestacks are stacks.
\end{definition}

\begin{remark}
\label{rem-twonewstacks}
Note that if we let $i$ denote the inclusion $Z \hookrightarrow X$, and $j$ the inclusion $X - Z \hookrightarrow X$, then in some sense $\scrC_Z$ is $i_! i^! \scrC$, and $\scrC_{X-Z}$ is $j_* j^* \scrC$.
In fact we even have an exact sequence
$$0 \to \scrC_Z \to \scrC \to \scrC_{X-Z}$$
in the sense that $\scrC_Z$ is a full, thick substack of $\scrC$, and it is exactly the ``kernel'' of the functor
$\scrC \to \scrC_{X-Z}$.  If $\scrC$ is of type P and $Z$ is a closed union of strata, then we may add a zero on the right as well, in the sense that the morphism $\scrC \to \scrC_{X-Z}$ exhibits each stalk of $\scrC_{X-Z}$ as a Serre quotient of the stalk of $\scrC$.
\end{remark}

\begin{theorem}
\label{th1holds}
Let $(X,\scrS)$ be a topologically stratified space.  Let $\scrC$ be a stack of type P, and let $Z$ be a closed union of strata.  Then the inclusion functor $\scrC_Z \to \scrC$ and the quotient functor $\scrC \to \scrC_{X - Z}$ have adjoints on both sides.
\end{theorem}

\begin{proof}
To show that $\scrC_Z \to \scrC$ has a left (resp. right) adjoint, it suffices to show that the induced functor on stalks $i_*:\scrC_{Z,x} \to \scrC_x$ has a left (resp. right) adjoint for each point $x \in X$.

Let $c$ be an object of $\scrC_x$.  Then $c$ has finite length by definition.  In particular $c$ contains a maximal subobject $c' \subset c$ supported on $Z$.  If $c_1$ is any other subobject supported on $Z$ then the compositum $c_1 + c'$ is supported on $Z$ and contains $c'$, so $c' = c_1 + c'$.  Thus $c_1 \subset c'$ and $c'$ is the unique maximal subobject of $c$ supported on $Z$.  For each $c\in \scrC_x$ let $i^! c$ be the unique maximal subobject of $c$ belonging to $\scrC_{Z,x}$.  This defines a left adjoint to $i_*$.
We may construct a right adjoint by a dual process.

Now let us show that $\scrC \to \scrC_{X-Z}$ has adjoints on both sides.  We may factor this functor as
$$\scrC \to \scrC_{X-Y} \to \scrC_{X-Z}$$
where $Y$ is a closed union of strata with one less stratum than $Z$, so by induction we are reduced to the case where $Z$ has only one stratum.  It again suffices to show that for $x \in X$ the functor $\scrC_x \to \scrC_{X-Z,x}$ has adjoints on both sides.  When $x \notin Z$ this is obvious, as $\scrC_x \to \scrC_{X-Z,x}$ is an equivalence.

Suppose $x \in Z$, and let us show that $j^*:\scrC_x \to \scrC_{X-Z,x}$ has a right adjoint.
We want to show that for each $c \in \scrC_{X-Z,x}$, the functor $d \mapsto \Hom(j^*d,c)$ is representable.  As $\Hom(j^*-,c)$ is left-exact and every object of $\scrC_{X-Z,x}$ has finite length, this functor is ind-representable by the system of objects $E_i$ equipped with isomorphisms $f_i:j^*E_i \cong c$.  If $c$ is an extension of $c''$ by $c'$, then the ind-object representing $\Hom(j^*-,c)$
has finite length provided that the ind-objects representing $\Hom(j^*-,c'')$ and $\Hom(j^*-,c')$ do, 
as we see from the exact sequence
$$0 \to \Hom(j^*-,c') \to \Hom(j^*-,c) \to \Hom(j^*-,c'')$$
Thus, we are reduced to showing that $\Hom(j^*-,c)$ is representable when $c$ is simple.

Suppose $c$ is simple.  Since $\scrC_{X-Z,x}$ is a Serre quotient of $\scrC_x$, there is a unique simple object $\tilde{c} \in \scrC_x$ with $j^* \tilde{c} \cong c$.  Let $s \in \scrC_x$ be the simple object supported on $Z$.  Let $n$ be the dimension of $\Ext^1(s,\tilde{c})$.  For $i=1,\ldots,n$ pick extension classes 
$$0 \to \tilde{c} \to Y_i \to s \to 0$$
that span $\Ext^1(s,\tilde{c})$.  Let $E$ be the categorical pushout of the $Y_i$ over $\tilde{c}$, so that
we have an exact sequence
$$0 \to \tilde{c} \to E \to s^{\oplus n} \to 0$$
Note that $\Hom(s,E) = 0$, since in the exact sequence
$$\Hom(s,\tilde{c}) \to \Hom(s,E) \to \Hom(s,s^{\oplus n}) \to \Ext^1(s,\tilde{c})$$
the first group is zero (as $s$ and $\tilde{c}$ are nonisomorphic simple objects) and the last map is an isomorphism by construction.

Now consider the hom set $\Hom(j^*d,c)$.  Since $\scrC_{X-Z,x}$ is a Serre quotient, an element $f:j^*d \to c$ may be represented by a right fraction
$$d \to E' \leftarrow \tilde{c}$$
where $\tilde{c} \to E'$ is an inclusion and $E'/\tilde{c}$ is supported on $Z$.  Since $\Ext^1(s,s) = 0$ we in fact have $E'/\tilde{c} \cong s^{\oplus m}$ for some $m$.  We claim that there is a unique map $E' \to E$ that carries $\tilde{c}$ identically onto $\tilde{c}$; indeed, the set of such maps is the inverse image of the given inclusion under $\Hom(E',E) \to \Hom(\tilde{c},E)$, but this map is injective since $\Hom(s^{\oplus m},E) = \Hom(s,E)^{\oplus m} = 0$.  It follows that $E$ represents $\Hom(j^*-,c)$, as desired.  

A dual proof shows that $j^*$ has a left adjoint as well.
\end{proof}

\section{Recollement of abelian categories}
\label{section4}

In this section we prove the second part of theorem \ref{maintheorem}: that a stack of type P looks locally like a MacPherson-Vilonen construction.  A major ingredient is the recollement formalism introduced in \cite{bbd}.  (We follow \cite{cps} in using ``recollement'' as an English word.)

Let $D',D,D''$ be triangulated categories.  In \cite{bbd}, it is shown how to construct a $t$-structure on $D$ from $t$-structures on $D'$ and $D''$ using ``recollement data'' on $D',D,D''$.  The recollement structure on the triangulated categories induces a kind of abelian-category recollement structure on the hearts 
$\scrA',\scrA,\scrA''$.  A simple formulation of this is strong enough for our purposes:

\begin{definition}
The data of abelian categories and functors
$\scrA' \stackrel{i_*}{\to} \scrA \stackrel{j^*}{\to} \scrA''$ is called a \emph{recollement triple}
if
\begin{enumerate}
\item $i_*$ and $j^*$ are exact functors and admit adjoints on both sides.
\item $i_*$ is an equivalence onto the full subcategory of $\scrA$ annihilated by $j^*$.
\item The functor $\scrA/\scrA' \to \scrA''$ induced by $j^*$ is also an equivalence.
\end{enumerate}
In this situation we will say that $\scrA$ is a \emph{recollement extension of $\scrA''$ by $\scrA'$}.
\end{definition}
We denote by $i^*$ the left adjoint and $i^!$ the right adjoint to $i_*$, and $j_!$ the left adjoint
and $j_*$ the right adjoint to $j^*$.  We also set $i_! =_\mathrm{def} i_*$ and $j^! =_\mathrm{def} j^*$. 
The adjunction morphisms provide natural maps $j_! \to j_*$ and $i^! \to i^*$.  

Sheaves and perverse sheaves form the original example of a recollement triple: If $(X,\scrS)$ is a topologically stratified space and $Z$ is a closed union of strata, then the restriction functor $\bP(X) \to \bP(X-Z)$ and the extension-by-zero functor $\bP(Z) \to \bP(X)$ exhibit $\bP(X)$ as a recollement extension of $\bP(X-Z)$ by $\bP(X)$.  Here $\bP$ can be the category of $\scrS$-constructible perverse sheaves of any fixed perversity.  In \cite{macphersonvilonen} another class of examples was introduced:

\begin{example}
\label{MVconstruction}
Suppose we are given a right exact functor $F:\scrA'' \to \scrA'$, a left exact functor $G:\scrA'' \to \scrA'$, and a natural transformation $T:F \to G$.  The MacPherson-Vilonen category $C = C(F,G;T)$ is the category of ``factorizations of $T$": its objects are tuples $(A'',A',m,n)$ where $m:F(A'') \to A'$, $n:A' \to G(A'')$, and $n \circ m = T_{A''}$.  The functors $j^*:(A'',A',m,n) \mapsto A''$ and $i_*:A' \mapsto (0,A',0,0)$
exhibit $C$ as a recollement extension of $\scrA''$ by $\scrA'$.
\end{example}

In \cite{macphersonvilonen}, MacPherson and Vilonen construct a functor $\bP(X) \to C(F,G;T)$, for certain $F,G:\bP(X -Z) \to \bP(Z)$. The functor respects all the recollement data, and it is an equivalence when $Z$ is contractible.  Since $C(F,G;T)$ is described pretty explicitly (up to a description of $\scrA'$ and $\scrA''$), this can be regarded as a ``calculation" of $\bP(X)$.  

One of the features of the MacPherson-Vilonen construction is that the inclusion $i_*:\scrA' \to C$
admits an \emph{exact} retraction: one may take $(A'',A',m,n) \mapsto A'$.  On the
other hand if a recollement triple $\scrA' \to \scrA \to \scrA''$ admits an exact retract 
$r:\scrA \to \scrA'$ one can use it to build a functor from $\scrA$ to a MacPherson-Vilonen category $C(F,G;T)$ for some $F$, $G$, and $T$.  Namely, take $F = r \circ j_!$ and $G = r \circ j_*$, and take $T$ to be the retract $r$ applied to the natural map $j_! \to j_*$.  Then
$$\scrA \to C(F,G;T):A \mapsto \big(F(j^*A) \to r(A) \to G(j^*A)\big)$$
is an exact functor.

In \cite{macphersonvilonen}, for perverse sheaves, a retract $\bP(X) \to \bP(Z)$ is built out of analytic or topological information, using vanishing cycles or perverse
links.  One point of the second part of theorem \ref{maintheorem} is that a retract exists for formal reasons:

\begin{lemma}
\label{retractions}
Let $\scrA$ be an $\bbF$-linear abelian category, such that every object has finite length, and such that each $\Hom(a,b)$ and each $\Ext^1(a,b)$ are finite-dimensional over $\bbF$.  
Let $s$ be a simple object in $\scrA$ with $\Ext^1(s,s) = 0$.
Then the recollement triple
$$\bbF\text{-Vect} \to \scrA \to \scrA/s$$
where the first functor maps $\bbF^n$ to $s^{\oplus n}$, admits an exact retraction.
\end{lemma}

\begin{proof}
We have to construct an exact functor $\scrA \to \bbF\text{-Vect}$ that takes $s$ to $\bbF$.  Gabriel constructed a pro-representable such functor in \cite{gabriel}, in greater generality, but we give a construction here as well.  An \emph{indecomposable cover} of $s$ is an object $x$
together with a surjection $x \to s$ that exhibits $s$ as the cosocle of $x$ -- \emph{i.e.} as the largest
semisimple quotient of $x$.  Such an $x$ is necessarily indecomposable.

If $y \to s$ is any surjection, it is easily verified that $y$ contains an indecomposable cover of $s$.
If $x$ is an indecomposable cover of $s$, it is also easily verified that any map from $y$ to $x$ over
$s$ is surjective.

We want to consider the limit of all indecomposable covers of $s$.  Let $C_s$ denote the category
of indecomposable covers of $s$, and let $\tilde{C}_s$ denote any \emph{filtered} category
with a functor $\tilde{C}_s \to C_s$ that is surjective on morphisms (\emph{i.e.} any morphism $f \in C_s$
is the value of some morphism $f \in \tilde{C}_s$).  Set 
$$P_s(a) = \varinjlim_{x \in \tilde{C}_s} \Hom(x,a)$$
Then $P_s(t) = 0$ for any simple $t \neq s$, since $s$ is the cosocle of each $x$, and $P_s(s) = \bbF$
since $\Hom(x,s) = \bbF$ for each $x$.  
It remains to show that $P_s$ is exact.  
Let $$0 \to a \to b \to c \to 0$$ be
a short exact sequence in $\scrA$; each $x \in C_s$ induces an exact sequence 
$$0 \to \Hom(x,a) \to \Hom(x,b) \to \Hom(x,c) \to \Ext^1(x,a)$$
Since, $\tilde{C}_s$ is filtered,
the sequence 
$$0 \to P_s(a) \to P_s(b) \to P_s(c) \to \varinjlim \Ext^1(x,a)$$ is exact as well.  To see that
$P_s$ is exact one needs to show the limit of $\Ext^1$s vanishes.   Suppose $y$ is an extension of $x$ by $a$;
we want to find an $x'$ mapping to $x$ such that the fiber product $y \times_x x' \to x'$
splits.  But we may take for $x'$ any indecomposable cover of $s$ contained in $y$.
\end{proof}

\begin{remark}
Note that the retraction $\scrA \to \bbF\text{-Vect}$ constructed has the further property that it kills simple objects not isomorphic to $s$.
\end{remark}

Define an additive functor 
$j_{!*}:\scrA'' \to \scrA$ by taking $a''$ to the image of $j_!{a''}$ in $j_*{a''}$.  We mention the following
proposition, from \cite{bbd}: %%Proposition 1.4.26

\begin{proposition}
\label{bbd1426}
The functor $j_{!*}$ preserves simples.  Furthermore, every simple in $\scrA$ is of the form $i_!s$
where $s$ is simple in $\scrA'$, or of the form $j_{!*}s$ where $s$ is simple in $\scrA''$.
\end{proposition}

Let $\scrA$ and $\scrB$ be two recollement extensions of $\scrA''$ by $\scrA'$.  If one is given an exact functor $R:\scrA \to \scrB$, together with natural isomorphisms $R \circ i_* \cong i_*$ and $j^* \circ R \cong j^*$, then using the adjunctions we may define maps $R\circ j_* \to j_*$, $j_!  \to R \circ j_!$, 
$i^* \circ R \to i^*$, and $i^! \to i^! \circ R$.

\begin{definition}
Let $\scrA' \stackrel{i_*}{\to} \scrA \stackrel{j^*}{\to} \scrA''$ and $\scrA' \stackrel{i_*}{\to} \scrB \stackrel{j^*}{\to} \scrA''$ be two recollement triples.  A \emph{morphism of recollement extensions} from $\scrA$ to $\scrB$ is an exact functor $R:\scrA \to \scrB$
together with isomorphisms $R \circ i_* \cong i_*$ and $j^* \circ R \cong j^*$, such that the induced
morphisms $R \circ j_* \to j_*$, $j_! \to R \circ j_!$, $i^* \circ R \to i^*$, and $i^! \to i^! \circ R$ are isomorphisms.
\end{definition}

Now let $\scrC$ be a stack of type P on $X$.  For each point $x \in X$, the stalk $\scrC_x$ has a unique
simple $s_x$ supported on the stratum containing $x$.  
Since $\Ext^1(s_x,s_x) = 0$, the recollement triple
$$\bbF\text{-Vect} \to \scrC_x \to \scrC_x/s_x$$
admits a retraction $r$ by lemma \ref{retractions}.  The retraction induces a functor
$R:\scrC_x \to C(F,G;T)$, where $F$ and $G$ map $\scrC_x/s_x$ to $\bbF\text{-Vect}$.  
Note that,
if $U$ is a regular neighborhood of $x$ and $Y$ is the stratum containing $x$,
then $\scrC_x/s_x \cong \scrC(U - U \cap Y)$.  

\begin{proposition}
\label{ff}
Let $(X,\scrS)$ be a topologically stratified space, and let $\scrC$ be a stack of type P on $X$.  Let $x \in X$ and $F,G$, and $R$ be as above.  The functor $R:\scrC_x \to C(F,G;T)$ is fully faithful.
\end{proposition}

\begin{proof}
The functor $R$ is clearly faithful.  We will show that it is full following the proof of proposition 1.2 in \cite{vilonenfda}.  We have to show that for $c$ and $d$ in $\scrC_x$ the map $\Hom(c,d) \to \Hom(Rc,Rd)$ is an isomorphism.  This is clear from the adjunctions in case $d$ is of the form $j_* d'$ for some $d' \in \scrC_x/s_x$, or of the form $s_x^{\oplus n}$.  In the latter case we furthermore have that $\Ext^1(c,s_x) \to \Ext^1(Rc,Rs_x)$ is injective, since if $0 \to s_x \to E \to c \to 0$ is an exact sequence which splits after applying $R$, then we may lift a retraction $RE \to R s_x$ to $E \to s_x$, using the fact that $\Hom(E,s_x) \cong \Hom(RE,Rs_x)$.

For a general $d$, consider the exact sequence
$0 \to K \to d \to j_* j^* d \to C \to 0$.
The objects $K$ and $C$ are direct sums of copies of $s_x$.  Let $I$ be the image of $d$ in $j_* j^* d$; since $\Hom(c,j_*j^* d) \cong \Hom(Rc,Rj_* j^* d)$ and $\Hom(c,C) \cong \Hom(Rc,RC)$ the map $\Hom(c,I) \to \Hom(Rc,RI)$ is also an isomorphism.  Now we conclude that $\Hom(c,d) \to \Hom(Rc,Rd)$ is surjective from the five lemma applied to the map of exact sequences:
$$
\xymatrix{
0 \ar[r] & \Hom(c,K) \ar[r] \ar[d]^{\cong} & \Hom(c,d) \ar[r] \ar[d] & \Hom(c,I) \ar[r] \ar[d]^{\cong} & \Ext^1(c,K) \ar[d]^{\text{inj.}} \\
0 \ar[r] & \Hom(Rc,RK) \ar[r] & \Hom(Rc,Rd) \ar[r]& \Hom(Rc,RI) \ar[r] & \Ext^1(Rc,RK)
}
$$
\end{proof}

\subsection{Tilting extensions}

To show that $R$ is essentially surjective we will follow the proofs of the first two propositions of \cite{bbm}.  We need to translate some of the tilting formalism into our context.
Let us say that an object $M$ of either $\scrC_x$ or $C(F,G;T)$ is \emph{tilting} if we have 
$\Ext^1(s_x,M) = 0$ and $\Ext^1(M,s_x) = 0$.

\begin{proposition}
\label{prop-tilting}
Let $c$ be an object of $\scrC_x/s_x$.  There is a tilting object $c^\tilt$ such that $Rc^\tilt$ is also tilting, and such that $j^* c^\tilt \cong j^*R c^\tilt \cong c$.
\end{proposition}

\begin{proof}
Consider the exact sequence
$$0 \to a \to j_! c \to j_* c \to b \to 0$$
in $\scrC_x$.  Since $\Ext^2(b,a) = 0$ we may find an object $c^\tilt$ containing $j_! c$ as a subobject and isomorphisms $c^\tilt/a \cong j_* c$ and $c^\tilt/j_! c \cong b$.  The short exact sequences
$$
\Ext^1(s_x,a)  \to \Ext^1(s_x,c^\tilt) \to  \Ext^1(s_x,j_* c)$$
and
$$
\Ext^1(j_! \overline{c},s_x)  \to \Ext^1(\overline{c}^\tilt,s_x) \to  \Ext^1(b,s_x) 
$$  
show that the middle $\Ext^1$-groups vanish and that $c^\tilt$ is tilting.  The analogous exact sequences in $C(F,G;T)$ show that $Rc^\tilt$ is tilting as well.
\end{proof}

\begin{theorem}
\label{th2holds}
Let $(X,\scrS)$ be a topologically stratified space, and let $\scrC$ be a stack of type P on $X$.  Let $x \in X$ and $F,$ $G$, $T$, and $R$ be as above.  The functor $R:\scrC_x \to C(F,G;T)$ is an equivalence of categories.
\end{theorem}

\begin{remark}
A similar theorem is the main result of \cite{franjoupirashvili}.  There, a criterion is given for a functor between two different recollement extensions of abelian categories to be an equivalence, and this criterion is essentially a weakened form of our axiom (3).  However, it is also assumed that the categories involved have enough projectives, which is not the case in our situation, so we cannot apply this result directly.
\end{remark}

\begin{proof}
After proposition \ref{ff} it remains to show that $R$ is essentially surjective.  For each $c \in \scrC_x/s_x$ let $\bE_1(c)$ be the category of lifts of $c$ to $\scrC_x$ and let $\bE_2(c)$ be the category of lifts of $c$ to $C(F,G;T)$.  To show that $R$ is essentially surjective it suffices to show that $R$ induces an equivalence of categories $\bE_1(c) \to \bE_2(c)$.

Let us fix a tilting object $c^\tilt \in \scrC_x$ as in proposition \ref{prop-tilting}.  Set $\Psi(c) = i^! c^\tilt$ and $\Psi'(c) = i^*c^\tilt$, let $\tau$ denote the natural map $\Psi \to \Psi'$, and let $\bF(c)$ be the category of factorizations $\Psi \to \Phi \to \Psi'$ of $\tau$.  We will construct equivalences $\bE_1(c) \cong \bF(c)$ and $\bE_2(c) \cong \bF(c)$ commuting with the functor $\bE_1 \to \bE_2$, which will imply that it is an equivalence.

For each $\tilde{c} \in \bE_1(c)$ consider the complex
$$K(\tilde{c}):=[\cdots 0 \to j_! c \to \tilde{c} \oplus c^\tilt \to j_* c \to 0 \cdots]$$
in $\scrC_x$ associated to the following commutative square, regarded as a double complex:
$$\xymatrix{
j_! c \ar[r] \ar[d] & c^\tilt \ar[d] \\
\tilde{c} \ar[r] & j_* c}
$$
In this square, the top arrow is injective and the right arrow is surjective; furthermore, all four maps are isomorphisms modulo $s_x$.  It follows that the cohomology of this complex is concentrated in degree zero, and is supported on the stratum containing $x$.  Define $\Phi(\tilde{c}) = H^0(K(\tilde{c}))$.  The morphism $\Psi \to \Phi$ is defined at the level of complexes:
$$\xymatrix{
0 \ar[r] \ar[d] & c^\tilt \ar[r] \ar[d]^{(0,id)} & j_* c \ar[d]^{=} \\
j_! c \ar[r] & \tilde{c} \oplus c^\tilt \ar[r] & j_* c }$$
and the morphism $\Phi \to \Psi'$ is defined similarly.

We may define a similar functor $\bE_2 \to \bF$ by using $Rc^\tilt$ in place of $c^\tilt$; the map $\bE_1 \to \bF$ is then isomorphic to the composition $\bE_1 \to \bE_2 \to \bF$.

Now let us define an inverse functor $\bF \to \bE_1$.  We will send $\Psi \stackrel{\alpha}{\to} \Phi \stackrel{\beta}{\to} \Psi'$ to a subquotient of $i_!\Phi \oplus c^\tilt$.  Namely, we send it to the cohomology of the complex $L(\Phi) = L(\Phi,\alpha,\beta)$, also obtained from a commutative square:
$$L(\Phi) :=[\Psi \to \Phi \oplus c^\tilt \to \Psi']$$
Note that $L(K(\tilde{c}))$ is the cohomology of the complex associated to the following double complex:
$$
\xymatrix{
0 \ar[r] \ar[d] & j_! c \ar[r]^{=} \ar[d] & j_! c \ar[d] \\
c^\tilt \ar[r] \ar[d] & \tilde{c} \oplus c^\tilt \oplus c^\tilt \ar[r] \ar[d] & c^\tilt \ar[d] \\
j_* c \ar[r]^{=} & j_* c \ar[r] & 0}$$
The top and bottom rows of this complex are acyclic, and the cohomology of the middle row is easily seen to be $\tilde{c}$.  A natural isomorphism $K(L(\Phi)) \cong \Phi$ may be constructed similarly, so the functor $\bE_1 \to \bF$ is an equivalence.  Replacing $c^\tilt$ with $Rc^\tilt$ in the above shows $\bE_2 \to \bF$ is an equivalence as well.  This completes the proof.
\end{proof}

\section{The category of global objects}
\label{secglobal}

In this section, we prove part (3) of theorem \ref{maintheorem}.  That is, we show that the category $\scrC(X)$ of global objects of a stack of type P, on a space with 2-connected strata, is equivalent to the category of modules over a finite-dimensional algebra.  The proof requires us to add an additional assumption to our stratification: that strata ``have regular neighborhoods'' as in the following definition.  I do not know whether this condition is satisfied for every topologically stratified space, but it is always satisfied in practice -- for instance, Thom-Mather stratifications come with such neighborhoods or ``tube systems'' by definition.

\begin{definition}
\label{regnbd}
Let $S$ be a stratum in a topologically stratified space $(X,\scrS)$.  A \emph{regular neighborhood} for $S$ is a tuple $(U,L,\pi)$, where $U \subset X$ is an open subset containing $S$, $L$ is another topologically stratified space, and $\pi: U \to S$ is a continuous map that exhibits $U$ as a fiber bundle with fiber $CL$, the open cone on $L$.  We furthermore assume that the transition functions for a local trivialization of $\pi$ can be taken to be stratum-preserving.

Let us say that $(X,\scrS)$ \emph{has regular neighborhoods} if every stratum of $X$ has a regular neighborhood.
\end{definition}

\begin{lemma}
\label{regnbdlem}
Let $(X,\scrS)$ be a topologically stratified space, and let $\scrC$ be an $\scrS$-constructible stack on $X$.  Let $S$ be a stratum of $X$ and let $(U,L,\pi)$ be a regular neighborhood of $S$.  Let $i$ denote the inclusion map $S \hookrightarrow X$.  Then there is a natural equivalence of stacks $i^* \scrC \cong \pi_* (\scrC\vert_U)$.
\end{lemma}

In particular, the category $(i^*\scrC)(S)$ of global objects of $i^* \scrC$ is equivalent to the category $\scrC(U) = \pi_* \scrC(S)$ of $\scrC$-objects defined over $U$.

\begin{proof}
If $S'$ is an open subset of $S$
then the restriction functors $\scrC(\pi^{-1}(S')) \to \scrC(V)$ assemble to a functor $\scrC(\pi^{-1}(S')) \to \twocolim_{V} \scrC(V)$, where $V$ runs through the open subsets of $\pi^{-1}(S')$ containing $S'$.  This defines a morphism from $\pi_* \scrC$ to the prestack $i_p^* \scrC$, which in turn maps naturally to $i^* \scrC$, the stackification of $i_p^* \scrC$.  To see that the composition of these morphisms is an equivalence it suffices to show that for each $s \in S$, the functor on stalks $(\pi_* \scrC)_s \to (i^* \scrC)_s$ is an equivalence (note that $(i^* \scrC)_s$ is naturally equivalent to $\scrC_s$).  Let $W\subset S$ be a contractible neighborhood of $s$.  Then $\pi^{-1}(W)$ is a conical neighborhood of $s$ in $U$.  It follows from theorem 3.13 of \cite{epcs} that we have $\scrC_s \cong \scrC(\pi^{-1}(W))$ and $(\pi_* \scrC)_s \cong \pi_* \scrC(W) = \scrC(\pi^{-1}(W))$.  This completes the proof.
\end{proof}

\begin{theorem}
\label{th3holds}
Let $(X,\scrS)$ be a topologically stratified space that has regular neighborhoods, and let $\scrC$ be a stack of type P on $X$.  Suppose that $(X,\scrS)$ has only finitely many strata, and that each stratum is 2-connected.  Then the category $\scrC(X)$ is equivalent to the category of finite-dimensional modules over a finite-dimensional $\bbF$-algebra.
\end{theorem}

\begin{proof}
We will induct on the number of strata.  Suppose first there is only one stratum.  Then the assumptions imply $\scrC$ is locally constant and $X$ is 2-connected.  By the main theorem of \cite{2monodromy} (see also theorem 5.7 of  \cite{epcs}) and the fact that $X$ is 2-connected, $\scrC$ is actually constant.  It follows that $\scrC(X)$ is equivalent to the stalk category $\scrC_x$ for some point $x \in X$.  By axiom (3) of definition \ref{defPstacks}, this category is equivalent to the category of $\bbF$-vector spaces, which is a category of modules over a finite-dimensional $\bbF$-algebra.

Now suppose that $(X,\scrS)$ has $n$ strata and that we have proved the theorem for all spaces with $\leq (n-1)$ strata.  Let $Z \subset X$ be a closed stratum of $X$, and let $(U,L,\pi)$ be a regular neighborhood of $Z$.  Since $Z$ is 2-connected and $i^* \scrC$ (resp. $\scrC_Z$) is locally constant on $Z$, we see that $i^* \scrC$ (resp. $\scrC_Z$) is constant and $i^*\scrC(Z) \cong \scrC_z$, (resp. $\scrC_Z(Z) \cong \scrC_{Z,z} \cong \bbF\text{-Vect}$) for a point $z \in Z$.  By lemma \ref{regnbdlem}, $i^*\scrC(Z) \cong \scrC(U)$.  Thus, the recollement triple
$$\scrC_Z(U) \stackrel{i_{U*}}{\to} \scrC(U) \stackrel{j_U^*}{\to} \scrC(U - Z)$$
is equivalent to the triple
$$\bbF\text{-Vect} \to \scrC_x \to \scrC_x/s_x$$
In particular, $i_{U*}$ has a retraction $r_U:\scrC(U) \to \scrC_Z(U)$ which induces an equivalence to a MacPherson-Vilonen construction, by theorem \ref{th2holds}.  The recollement triple
$$\scrC_Z(X) \to \scrC(X) \to \scrC(X-Z)$$ 
also has a retraction $r_X$ given by $\scrC(X) 
\stackrel{\mathit{res}}{\to} \scrC(U) 
\stackrel{r}{\to} \scrC_Z(U) = \scrC_Z(X) = \scrC(X)$, and we will show that $r_X$ induces an equivalence to a MacPherson-Vilonen construction as well.

Set $F_U = r_U (j_{U!})$, $G_U = r_U(j_{U*})$, let $T_U$ be the natural map $F_U \to G_U$, and let $R_U$ denote the functor $\scrC(U) \to C(F_U,G_U;T_U)$ induced by $r_U$.  Define $F_X, G_X, T_X$ and $R_X:\scrC(X) \to C(F_X,G_X;T_X)$ similarly.  By theorem \ref{th2holds},  $R_U$ is an equivalence.
The recollement triple
$$\scrC_Z(X) \to \scrC(X) \to \scrC(X-Z)$$ 
also has a retraction $r_X$ given by $\scrC(X) \stackrel{\mathit{res}}{\to} \scrC(U) \stackrel{r}{\to} \scrC_Z(U) = \scrC_Z(X) = \scrC(X)$.  The functor $r_X$ induces a functor $\scrC(X) \to C(F_X,G_X;T_X)$.  

Since $\scrC$ is a stack, we have a weak pullback square of categories:
$$\xymatrix{
 \scrC(X) \ar[r] \ar[d] & \scrC(X-Z) \ar[d] \\
  \scrC(U) \ar[r] & \scrC(U-Z)
}$$
Concretely, this means $\scrC(X)$ is equivalent to the 
category of triples $(c_1,c_2,\theta)$, where $c_1 \in \scrC(U)$, $c_2 \in \scrC(X-Z)$, and $\theta$ is an isomorphism between $c_1 \vert_{U-Z}$ and $c_2 \vert_{U-Z}$.  To show $R_X$ is an equivalence, it suffices to show that the square
$$\xymatrix{
C(F_X,G_X;T_X) \ar[r] \ar[d] \ar[r] \ar[d] & \scrC(X-Z) \ar[d] \\
C(F_U,G_U;T_U) \ar[r] & \scrC(U-Z)
}$$
is also a weak pullback.  That is, we must show that $C(F_X,G_X;T_X)$ is equivalent to the category of triples $(M,b_2,\theta)$, where 
\begin{enumerate}
\item[] $M = (F_U(b_1) \stackrel{m}{\to} V \stackrel{n}{\to} G_U(b_1))$ is an object of $C(F_U,G_U;T_U)$
\item[] $b_2$ is an object of $\scrC(X-Z)$
\item[] $\theta$ is an isomorphism  between $b_1$ and $b_2 \vert_{U-Z}$.
\end{enumerate}
But we have $F_X(b_2) = F_U(b_2\vert_{U-Z}) \stackrel{\theta}{\cong} F_U(b_1)$ and $G_X(b_2) = G_U(b_2\vert_{U-Z}) \stackrel{\theta}{\cong} G_U(b_1)$ by definition, so the assignment $(M,b_2,\theta) \mapsto (F_X(b_2) \to V \to G_X(b_2))$ is an equivalence of categories.

By the inductive hypothesis $\scrC(X)$ is equivalent to a category built from a category of finite-dimensional modules over a finite-dimensional $\bbF$-algebra $A$ via a $C(F,G;T)$-construction, with $F,G:A\text{-mod} \to \bbF\text{-Vect}$.  By construction every object of this category has finite length, and every hom set is a finite-dimensional $\bbF$-vector space.  To show that such a category is the category of finite-dimensional modules over a finite-dimensional $\bbF$-algebra it suffices to show that it has enough projectives, and this follows from proposition 2.5 of \cite{mirollovilonen}.  This completes the proof.
\end{proof}

\emph{Acknowledgements:}  I would like to thank Mark Goresky, Bob MacPherson, Andrew Snowden, Kari Vilonen, and Zhiwei Yun.  This paper has benefited from their comments.


\begin{thebibliography}{99}
\bibitem{bbd} A. Beilinson, J. Bernstein, and P. Deligne, \emph{Faisceaux pervers}.  Ast\'erisque 100 (1982) pp. 5--171.
\bibitem{bbm} A. Beilinson, R. Bezrukavnikov, and I. Mirkovic, \emph{Tilting exercises}. Mosc. Math. J. 4 (2004) pp. 547--557. 
\bibitem{braden}  T. Braden.  \emph{Perverse sheaves on Grassmannians.}  Canad. J. Math. 54 (2002) pp. 493-532.
\bibitem{bradengrinberg}  T. Braden and M. Grinberg. \emph{Perverse sheaves on rank stratifications.}  Duke Math. J. 96 (1999) pp. 317-362.

\bibitem{cps}  E. Cline, B. Parshall, and L. Scott. \emph{Finite-dimensional algebras and highest weight categories.} J. Reine Angew. Math. 391 (1988) pp. 85-99.
\bibitem{franjoupirashvili} V. Franjou and T. Pirashvili, \emph{Comparison of abelian categories recollements}.  Doc. Math. 9 (2004), pp. 41-56.
\bibitem{gabriel} P. Gabriel, \emph{Des categories abeliennes.}  Bull. Soc. Math. France 90 (1962) pp. 323-448.

\bibitem{gmvquiver} S. Gelfand, R. MacPherson,  and K. Vilonen, \emph{Perverse sheaves and quivers.} Duke Math. J. 83 (1996) pp. 621-643.

\bibitem{ih2} M. Goresky and R. MacPherson.  \emph{Intersection homology II.}  Invent. Math. 72 (1983) pp. 77-129.
\bibitem{macphersonvilonen} R. MacPherson and K. Vilonen, \emph{Elementary constructions of perverse sheaves.} Invent. Math. 84 (1986) pp. 403-435.
\bibitem{mirollovilonen} R. Mirollo and K. Vilonen \emph{Bernstein-Gelfand-Gelfand reciprocity on perverse sheaves.}   Ann. Sci. \'Ecole Norm. Sup. 20 (1987) pp. 311-323.
\bibitem{2monodromy} P. Polesello and I. Waschkies, \emph{Higher monodromy.}  Homology Homotopy Appl. 7 (2005), pp. 109-150

\bibitem{epcs} D. Treumann, \emph{Exit paths and constructible stacks}.  Preprint available at {\sf http://arxiv.org/abs/0708.0659}
\bibitem{vilonenfda} K. Vilonen \emph{Perverse sheaves and finite-dimensional algebras}.  Trans. Amer. Math. Soc. 341 (1994), pp. 66--676.

\end{thebibliography}
\end{document}